%% file: main.tex
\pdfoutput=1
\documentclass[oneside]{amsart}

\newcommand\tab[1][0.5cm]{\hspace*{#1}}

\usepackage[a4paper, total={6in, 9in}]{geometry}

\title[Cubic Chabauty for hyperelliptic curves]{A cubic Chabauty criterion for hyperelliptic curves}
\author{Lee Berry}
\address{Department of Mathematics, King’s College London, Strand, London, WC2R 2LS, UK}
\email{lee.berry@kcl.ac.uk}

\date\today
\input{header}
\begin{document}

\begin{abstract}
We establish a criterion for finiteness of the depth 3 Chabauty--Kim set $X(\mathbb{Q}_2)_3$ for hyperelliptic curves $X$ over $\mathbb{Q}$ with a rational Weierstrass point. To illustrate its applicability, we include an example of a genus $2$ curve satisfying the criterion whose Jacobian has N\'{e}ron--Severi rank $1$ and Mordell--Weil rank $4$.
\end{abstract}

\maketitle


\section{Introduction}

Let $X$ be a smooth, projective, geometrically integral curve over $\mathbb{Q}$ of genus $g\geq 2$, and let $p$ be a prime of good reduction for $X$. The \textit{Chabauty--Kim method} \cite{kim2005motivic,kim2009unipotent} equips the set of rational points $X(\mathbb{Q})$ with a tower of approximations of $p$-adic points: 
\begin{align*}
X(\mathbb{Q})\subseteq \cdots \subseteq X(\mathbb{Q}_p)_n\subseteq \cdots \subseteq X(\mathbb{Q}_p)_2\subseteq X(\mathbb{Q}_p)_1 \subseteq X(\mathbb{Q}_p).
\end{align*}
Although $X(\mathbb{Q})$ is finite by Faltings' theorem, \textit{Kim's conjecture} \cite[§3]{kim2009unipotent} asserts that, in fact, the \textit{depth $n$ Chabauty--Kim set} $X(\mathbb{Q}_p)_n$ is finite for all sufficiently large $n$. Accordingly, exploiting this setup to compute $X(\mathbb{Q})$ typically requires both establishing a finiteness criterion for $X(\mathbb{Q}_p)_n$ at some depth $n$ and a method to compute this set effectively. Fortunately, there exist such criteria, and corresponding methods of effective computation, for depths 1 and 2 sets due to Chabauty--Coleman \cite{Chab,colemanint} and Balakrishnan--Dogra \cite{Balakrishnan_2018}, respectively. However, at present there are no equally general analogous results for depth $n\geq 3$. Even in cases where Kim’s conjecture is known (for instance, it holds for hyperelliptic curves \cite{EllenbergHast}), these results prove only eventual finiteness in the above tower and do not establish criteria for finiteness at a given depth. 

Even at depth $n=3$, the Bloch--Kato conjectures \cite{bloch1990functions} predict a significant extension of the Chabauty--Kim method beyond its known effective instances at depths 1 and 2. For example, in the case $g=2$, these conjectures imply that $X(\mathbb{Q}_p)_3$ is generically finite whenever the Mordell--Weil rank of the Jacobian of $X$ is less than 12. By contrast, $X(\mathbb{Q}_p)_2$ is generically expected to be finite in this case only for rank less than 4. Motivated by this conditional applicability, in this paper we develop a criterion for finiteness of the depth 3 set $X(\mathbb{Q}_2)_3$ for hyperelliptic curves of arbitrary genus with a rational Weierstrass point. 

To deduce finiteness of $X(\mathbb{Q}_p)_n$, we study global \textit{Bloch--Kato Selmer groups} $H^1_f(\mathbb{Q},V_i)$ associated to certain $p$-adic Galois representations $V_i$ ($1\leq i\leq n$) defined by the fundamental group of $X$. These groups naturally arise from the definition of $X(\mathbb{Q}_p)_n$ and, in particular, a general theorem of Kim asserts that the desired finiteness follows from sufficiently sharp bounds for the $\mathbb{Q}_p$-dimensions of these Selmer groups (see Theorem \ref{KimThm} below). For instance, in the depth 2 case, obtaining a finiteness criterion requires bounding the dimension of $H^1_f(\mathbb{Q},\wedge^2 V_pJ)$, where $V_pJ$ denotes the $\mathbb{Q}_p$-rational Tate module of the Jacobian $J$ of $X$. Bounding this dimension in general is highly nontrivial and is related, via the Bloch--Kato conjectures, to obtaining the rank of a certain higher Chow group associated to $X$. As such, the \textit{quadratic Chabauty method for rational points} \cite{Balakrishnan_2018} circumvents the necessity to compute such bounds, instead applying Kim's theorem to establish finiteness of a set lying between $X(\mathbb{Q}_p)_1$ and $X(\mathbb{Q}_p)_2$. Specifically, this method guarantees that this intermediate set is finite whenever the rank of $J(\mathbb{Q})$ is less than $g+\rho-1$, where $\rho$ denotes the rank of the N\'{e}ron--Severi group of $J$.  

In the case that $X$ is a hyperelliptic curve, the recent development of 2-descent methods for Bloch--Kato Selmer groups \cite{bk1}, coupled with subsequent papers of Dogra \cite{bk2} and the author \cite{1stpaper}, yields effective bounds for the desired dimension of $H^1_f(\mathbb{Q},\wedge^2V_2J)$. As a consequence, this has yielded many new examples of curves, whose Jacobians have N\'{e}ron--Severi rank 1, for which finiteness of $X(\mathbb{Q}_2)_2$ has been verified. In this paper, we exploit this 2-descent machinery to derive an analogous criterion for depth 3 sets.

For the remainder of the introduction, fix $p=2$ and assume that $X$ is a hyperelliptic curve over $\mathbb{Q}$, defined as the smooth projective completion of the affine curve $y^2=f(x)$ such that the degree of $f\in \mathbb{Q}[x]$ is odd. In the depth 3 case, it is necessary to bound the $\mathbb{Q}_2$-dimension of the Bloch--Kato Selmer group $H^1_f(\mathbb{Q},V_3)$,  where $V_3$ is the cokernel of a natural map
\begin{align*}
{\wedge^3 V_2 J} \rightarrow V_2J\otimes\big(({\wedge^2 V_2J)/\mathbb{Q}_2(1)}\big),
\end{align*}
the definition of which will be given in Section \ref{CKSection}. Prior to stating the main theorem describing this bound, we introduce additional notation required for a particularly concise case of the main result.

Specifically, this case requires additional assumptions on the defining polynomial $f$ of $X$. Firstly, assume that $f$ is irreducible. Next, let $\alpha,\beta,\gamma \in \overline{\mathbb{Q}}$ denote three distinct roots of $f$, and define the number fields 
\begin{align*}
K_1\coloneqq \mathbb{Q}(\alpha), \tab  K_2\coloneqq \mathbb{Q}(\alpha+\beta+\gamma), \tab K_3\coloneqq \mathbb{Q}(\alpha,\beta+\gamma).
\end{align*}
For a given $K_i$, if $S$ is a finite set of primes in $\mathbb{Q}$ containing $\infty$, then define 
\begin{align*}
d_i\coloneqq \mathrm{rk}_{\mathbb{Z}}\mathcal{O}_{K_i,S}^\times +\dim_{\mathbb{F}_2}\mathrm{Cl}_S(K_i)[2],
\end{align*}
where $\mathcal{O}_{K_i,S}$ denotes the ring of $S$-integers (that is, the ring of elements that are integral away from the primes above $S$) and $\mathrm{Cl}_S(K_i)\coloneqq \mathrm{Cl}(\mathcal{O}_{K_i,S})$ denotes the $S$-class group of $K_i$; we omit explicit reference to $S$ in the notation $d_i$ for simplicity. Finally, assume that 
\begin{align}\label{starassumption}\tag{$\star$}
[K_2:\mathbb{Q}]=\frac{1}{3}g(4g^2-1), \tab\phantom{} [K_3:K_2]=3. 
\end{align}
A similar statement of the main theorem holds without this final assumption, although we would then have to modify the definitions of each $K_i$. 

We now state the main result, which details the desired effective bounds and their implications for a \textit{cubic Chabauty} criterion for hyperelliptic curves. 

\begin{introtheorem}[Theorem \ref{mainthm}, Proposition \ref{MainApprox}]\label{ThmA}
Let $X$ be a hyperelliptic curve over $\mathbb{Q}$ of genus $g \geq 2$ defined by the equation $y^2=f(x)$, where $f\in \mathbb{Q}[x]$ is a separable polynomial of odd degree. Let $S$ be a finite set of primes containing $\{2,\infty\}$ and the primes of bad reduction for $X$. 
\begin{itemize}
    \item[(1)] For notational simplicity, denote
    \begin{align}\label{cubicintroineq}
    d_{X,S}\coloneqq \dim_{\mathbb{F}_2} H^1(G_S,(J[2] \otimes (\wedge^2J[2])/\mathbb{F}_2)/{\wedge^3J[2]}),
    \end{align}
    where $J[2]$ denotes the $2$-torsion subgroup of the Jacobian of $X$, and $G_S$ denotes the Galois group of the maximal extension of $\mathbb{Q}$ that is unramified outside of $S$. Then $d_{X,S}$ is effectively computable and $\dim_{\mathbb{Q}_2}H^1_f(G_S,U[3])<d_{X,S}$. 
    \item[(2)] If, in addition, the defining polynomial $f$ is irreducible and condition \hypersetup{linkcolor=black}\eqref{starassumption}\hypersetup{linkcolor=blue} holds, then
    \begin{align*}
    d_{X,S}=d_3-(d_2+d_1)+\#S.
    \end{align*}
    \item[(3)] Suppose that the following inequality holds:
    \begin{align*}
    d_{X,S}+\dim_{\mathbb{Q}_2}H^1_f(G_S,\wedge^2 V_2J)<\frac{1}{6}(g+1)(2g+1)(7g-6)-\mathrm{rk}_{\mathbb{Z}}J(\mathbb{Q}).
    \end{align*}
    Then the depth 3 Chabauty--Kim set $X(\mathbb{Q}_2)_3$ is finite. 
\end{itemize}
\end{introtheorem}

The assertion of effective computability in part (1) follows from $d_{X,S}$ generally admitting an analogous description to part (2) in terms of relevant $S$-unit and $S$-class groups; we reiterate that the additional assumptions are necessary only for expressing $d_{X,S}$ compactly.

To demonstrate the practical applicability of this cubic Chabauty criterion, we verify it for the following high-rank curve: 

\begin{introtheorem}[Theorem \ref{rank4thm}]\label{ThmB}
Let $X$ be the genus $2$, Mordell--Weil rank $4$ hyperelliptic curve over $\mathbb{Q}$ defined as the smooth projective completion of the affine curve
\begin{align*}
y^2=4x^5 + 8x^4 + 4x^2 + 8x + 1.
\end{align*}
Then the depth $3$ Chabauty--Kim set $X(\mathbb{Q}_2)_3$ is finite.
\end{introtheorem}

Note that the Bloch--Kato conjectures imply that the depth 2 set of this example is not expected to be finite. In fact, to the best of the author's knowledge, this is only the second example of a verified finite depth $n\geq 2$ set for a genus 2 curve with a rank $r\geq 3$ simple Jacobian (the first is contained in \cite{bk1}). 

\subsection{Notation and conventions}\label{Notation}

Let $K$ be a field, let $K^{\mathrm{sep}}$ be a fixed separable closure of $K$, and let $S$ be a finite set of primes in $K$ that contains the archimedean primes. Denote by $G_K$ (resp. $G_{K,S}$) the Galois group $\mathrm{Gal}(K^{\mathrm{sep}}/K)$ (resp. the Galois group of the maximal extension of $K$ unramified outside of $S$). For notational simplicity, we denote $G_S\coloneq G_{\mathbb{Q},S}$. 

If $K$ is a number field, $V$ is a finite-dimensional $\mathbb{Q}_p$-vector space with a continuous action of the absolute Galois group $G_K$, and $v$ is a prime in $K$, then the \textit{local Bloch--Kato Selmer group} $H^1_f(K_v,V)$ is the subspace of $H^1(K_v,V)$ defined as follows:
\begin{align*}
H^1_f(K_v,V)\coloneqq \begin{cases}
\mathrm{ker}\big(H^1(K_v,V)\rightarrow H^1(K_v,V \otimes B_{\mathrm{cris}})\big), &\text{if $v|p$} \\
H^1_{\mathrm{nr}}(K_v,V), &\text{otherwise.}
\end{cases}
\end{align*}
Subsequently, the \textit{global Bloch--Kato Selmer group} is defined by
\begin{align*}
H^1_f(K,V)\coloneqq \bigcap_v \mathrm{ker} \big(H^1(K,V)\rightarrow H^1(K_v,V)/H^1_f(K_v,V)\big).
\end{align*}
Moreover, denote by $D_{\mathrm{dR}}(V)$ the subspace $H^0(K,V\otimes B_{\mathrm{dR}})$ of $G_K$-invariant elements in $V\otimes B_{\mathrm{dR}}$; this subspace is equipped with a separated and exhaustive decreasing filtration, which we denote by $\mathrm{Fil}^\bullet D_{\mathrm{dR}}(V)$. 

Suppose $M=\prod_{i,j}M_{i,j}$ and $L=\prod_{i}L_i$ are finite \'{e}tale $K$-algebras such that each $M_{ij}/L_i$ is an extension of fields. For $A=(A_{ij})$ a tuple of $G_{M_{ij}}$-modules, we define 
\begin{align*}
\mathrm{Ind}_L^MA\coloneqq \bigoplus_{i,j}\mathrm{Ind}^{G_{M_{ij}}}_{G_{L_i}}A_{ij}. 
\end{align*}
For a separable polynomial $f\in K[x]$, we define the \'{e}tale algebra $K_f\coloneqq K[x]/(f)$. We will frequently exploit the well-known isomorphism of $\mathbb{F}_2[G_K]$-modules between $\mathrm{Ind}_K^{K_f}\mathbb{F}_2$ and the free $\mathbb{F}_2$-vector space generated by the roots of $f$.

We define an \textit{odd hyperelliptic curve} $X$ over $K$ as the smooth projective completion of the affine curve $y^2=f(x)$ such that the degree of the \textit{defining polynomial} $f\in K[x]$ is odd. Such a model exists if and only if $X$ has a rational Weierstrass point over $K$. We will often refer to $f$ as the defining polynomial of such a model, without explicitly restating that $X$ is given by $y^2=f(x)$. 

Finally, for an abelian group $N$ and an integer $k>0$, denote 
\begin{align*}
\wedge^k N \coloneqq N^{\otimes k}/I_k,
\end{align*}
where $I_k$ is generated by the classes $n_1\otimes \cdots \otimes n_k\in N^{\otimes k}$ such that $n_i=n_j$ for some $i\neq j$.

\subsection*{Acknowledgements} I am particularly indebted to my supervisor, Netan Dogra, for many invaluable discussions regarding explicit descent and the Chabauty--Kim method. I am also grateful to Corijn Rudrum and Giorgio Navone for helpful discussions. Finally, I am grateful to the examiners of my PhD thesis, Vladimir Dokchitser and Adam Morgan, for several suggestions that led to improvements to this paper. This work was supported by the Royal Society [RF\textbackslash ERE\textbackslash 210020] and a PhD studentship from the Faculty of Natural, Mathematical and Engineering Sciences at King's College London.

\section{Explicit descent and the Chabauty--Kim method}\label{CKSection}

Let $X$ be a smooth, projective, geometrically integral curve over $\mathbb{Q}$ of genus $g\geq 2$ and fix a prime $p$ of good reduction for $X$; we do not yet impose that $X$ is hyperelliptic. Denote by $\pi_1^{\mathbb{Q}_p}(X_{\overline{\mathbb{Q}}},b)$ the $\mathbb{Q}_p$\textit{-pro-unipotent \'{e}tale fundamental group} of $X$ based at a rational point $b\in X(\mathbb{Q})$. This fundamental group can, for instance, be defined as the continuous $\mathbb{Q}_p$-Mal\u{c}ev completion of the geometric \'{e}tale fundamental group $\pi_1^{\textrm{\'{e}t}}(X_{\overline{\mathbb{Q}}},b)$. Denote by $U^{(n)}$ the $n^{\textrm{th}}$ term in the descending central series filtration of $\pi_1^{\mathbb{Q}_p}(X_{\overline{\mathbb{Q}}},b)$, indexed by integers $n\geq 0$, such that $U^{(0)}\coloneqq\pi_1^{\mathbb{Q}_p}(X_{\overline{\mathbb{Q}}},b)$. Finally, define 
\begin{align*}
U_n\coloneqq \pi_1^{\mathbb{Q}_p}(X_{\overline{\mathbb{Q}}},b)/U^{(n)}, \tab U[n]\coloneqq \mathrm{ker} (U_n\rightarrow U_{n-1}).
\end{align*} 

Fundamental to the Chabauty--Kim method is the following commutative square:
\begin{center}
\begin{tikzcd}
X(\mathbb{Q}) \arrow{r} \arrow{d}{j_n} & X(\mathbb{Q}_p)\arrow{d}{j_{n,p}} \\ 
\mathrm{Sel}(U_n) \arrow{r}{\mathrm{loc}_{n,p}} & H^1(\mathbb{Q}_p,U_n),
\end{tikzcd}
\end{center}
from which we define the \textit{depth $n$ Chabauty--Kim set} as
\begin{align*}
X(\mathbb{Q}_p)_n\coloneqq j_{n,p}^{-1} \big(\mathrm{loc}_{n,p}(\mathrm{Sel}(U_n))\big).
\end{align*}

We now define the constituents of this square. We denote by $S$ a set of primes that contains $\{p,\infty\}$ and the primes of bad reduction for $X$. By foundational results in \cite{kim2005motivic,kim2009unipotent}, the non-abelian cohomology (pointed) sets $H^1(G,U_n)$, for $G \in \{G_S,G_{\mathbb{Q}_\ell}\}$ and $\ell$ any prime, are each isomorphic to the $\mathbb{Q}_p$-points of an affine algebraic variety over $\mathbb{Q}_p$. Moreover, if we define the subscheme $H^1_f(\mathbb{Q}_\ell,U_n)\subseteq H^1(\mathbb{Q}_\ell,U_n)$, which parametrises crystalline $G_{\mathbb{Q}_\ell}$-equivariant $U_n$-torsors (in the sense of the discussion preceding \cite[Lemma 5]{kim2009unipotent}), then Kim showed that it is itself the $\mathbb{Q}_p$-points of a variety. 

Subsequently, following \cite{Balakrishnan_2018}, we define the global \textit{Selmer scheme} of $U_n$ as the reduced subscheme of $H^1(G_S,U_n)$ consisting of classes $\xi \in H^1(G_S,U_n)$ such that: 
\begin{itemize}
    \item $\mathrm{loc}_{n,\ell}(\xi)\in j_{n,\ell}(X(\mathbb{Q}_\ell))$ for all primes $\ell\neq p$,
    \item $\mathrm{loc}_{n,p}(\xi)\in H^1_f(\mathbb{Q}_p,U_n)$,
    \item the image of $\xi$ in $H^1(G_S,U_1)\simeq H^1(G_S,V_pJ)$ is contained in the image of $J(\mathbb{Q})\otimes \mathbb{Q}_p$,
\end{itemize}
where $V_pJ$ will continue to denote the $\mathbb{Q}_p$-rational Tate module of $J$. The final requirement in this definition makes implicit use of the map from $H^1(G_S,U_n)$ to $H^1(G_S,U_1)$ under the well-known identification $U_1 \simeq V_pJ$. 

We denote by $j_n$ the \textit{pro-unipotent Kummer map}, which sends a rational point $x\in X(\mathbb{Q})$ to the $G_S$-equivariant $U_n$-torsor defined by pushing out the path torsor $\pi_1^{\mathbb{Q}_p}(X_{\overline{\mathbb{Q}}};b,x)$ along the quotient map that defines $U_n$. For each prime $\ell$, we denote by $j_{n,\ell}$ the local analogue of $j_n$, mapping a $\mathbb{Q}_\ell$-point to the corresponding $G_{\mathbb{Q}_\ell}$-equivariant torsor. Finally, we denote by $\mathrm{loc}_{n,p}$ the localisation map at $p$, which, by \cite[Proposition 2]{kim2005motivic}, defines an algebraic morphism of schemes $H^1(G_S,U_n) \rightarrow H^1(\mathbb{Q}_p,U_n)$.

The following result is the general theorem of Kim alluded to in the introduction, which underlies each of the finiteness criteria mentioned in this paper.

\begin{theorem}[\texorpdfstring{\cite[Theorem 1]{kim2009unipotent}}{}]\label{KimThm}
Suppose that 
\begin{align}\label{SelDimIneq}
\dim\mathrm{Sel}(U_n)<\dim H^1_f(\mathbb{Q}_p,U_n).
\end{align}
Then $X(\mathbb{Q}_p)_n$ is finite. 
\end{theorem}

The dimension on the right-hand side of \eqref{SelDimIneq}, as well as a bound for the dimension on the left-hand side, can be computed via Bloch--Kato Selmer groups associated to the graded pieces $U[n]$. For instance, we have the inequality 
\begin{align*}
\dim \mathrm{Sel}(U_n) \leq\mathrm{rk}_\mathbb{Z}J(\mathbb{Q})
+\sum_{i=2}^n \dim_{\mathbb{Q}_p}H^1_f(\mathbb{Q},U[i]),
\end{align*}
as stated in the proof of \cite[Lemma 5.2]{chabnumberfields}, where it is attributed to \cite[Propositions 1.1, 1.2 and 1.3]{tangential}. In particular, inequality \eqref{SelDimIneq} holds whenever 
\begin{align}\label{GeneralDimIneq}
\mathrm{rk}_{\mathbb{Z}}J(\mathbb{Q})+\sum_{i=2}^n\dim_{\mathbb{Q}_p}H^1_f(\mathbb{Q},U[i])<\sum_{i=1}^n\dim_{\mathbb{Q}_p}H^1_f(\mathbb{Q}_p,U[i]),
\end{align}
which will provide a more practical means of verifying the conditions of Theorem \ref{KimThm}. 

Accordingly, to obtain a cubic Chabauty criterion it suffices to consider exclusively the dimensions of the Bloch--Kato Selmer groups associated to the graded pieces $U[i]$ for $i\leq 3$. Firstly, the $i=1$ case is straightforward given the above isomorphism $U[1]=\pi_1^{\mathbb{Q}_p}(X_{\overline{\mathbb{Q}}},b)^{ab}\simeq V_pJ$. Next, for $i=2$ we have the split short exact sequence
\begin{align}\label{quadraticDesc}
0\rightarrow \mathbb{Q}_p(1) \rightarrow \wedge^2 V_pJ \rightarrow U[2] \rightarrow 0,
\end{align}
defined by the dual of the cup product map 
\begin{align}\label{cupprod}
\cup:\wedge^2H^1_{\textnormal{\'{e}t}}(X_{\overline{\mathbb{Q}}},\mathbb{Q}_p)\rightarrow H^2_{\textnormal{\'{e}t}}(X_{\overline{\mathbb{Q}}},\mathbb{Q}_p)\simeq \mathbb{Q}_p(-1).
\end{align} 
This sequence can, for instance, be deduced from the standard presentation of a surface group on $2g$ generators, given that $\pi_1^{\mathbb{Q}_p}(X_{\overline{\mathbb{Q}}},b)$ is the Mal\u{c}ev completion of the topological fundamental group of $X(\mathbb{C})$.

For $i\geq 3$, a description of $U[i]$ follows from \cite[Corollary 8.3]{Hain1995InfinitesimalPO}, under the isomorphism between $U[i]$ and the graded pieces of the corresponding Lie algebra described there. In particular, in the $i=3$ case we obtain the following short exact sequence:
\begin{align}\label{cubicDesc}
0 \rightarrow {\wedge^3 V_pJ} \rightarrow V_pJ \otimes U[2] \rightarrow U[3] \rightarrow 0,
\end{align}
which is induced by the map 
\begin{align*}
{\wedge^3 V_pJ} &\rightarrow V_pJ\otimes \wedge^2 V_pJ \\
v_1\wedge v_2 \wedge v_3 &\mapsto v_1\otimes (v_2\wedge v_3)+v_2\otimes (v_3\wedge v_1)+v_3\otimes (v_1\wedge v_2).
\end{align*}
With these explicit descriptions, we now state the $n=3$ instance of Theorem \ref{KimThm}.  

\begin{lemma}\label{netancor3}
Suppose that the following inequality holds:
\begin{align*}
\dim_{\mathbb{Q}_p} H^1_f(\mathbb{Q},U[2])+\dim_{\mathbb{Q}_p}H^1_f(\mathbb{Q},U[3]) < \frac{1}{6}(g+1)(2g+1)(7g-6) -\mathrm{rk}_\mathbb{Z}J(\mathbb{Q}).
\end{align*}
Then $X(\mathbb{Q}_p)_3$ is finite. 
\end{lemma}
\begin{proof}
By the inequality \eqref{SelDimIneq}, it suffices to compute the $\mathbb{Q}_p$-dimensions of the local Bloch--Kato Selmer groups $H^1_f(\mathbb{Q}_p,U[i])$ for $i\leq 3$. To that end, since each $U[i]$ is pure of negative weight, we may exploit the isomorphism
\begin{align*}
H^1_f(\mathbb{Q}_p,U[i])\simeq D_{\mathrm{dR}}(U[i])/\mathrm{Fil}^0D_{\mathrm{dR}}(U[i]),
\end{align*}
given that computation of the $\mathbb{Q}_p$-dimension of the right-hand side follows straightforwardly from the descriptions of $U[i]$ above. For $i\leq 2$, this calculation is essentially the proof of \cite[Lemma 25]{bk1}, which asserts that
\begin{align*}
\dim_{\mathbb{Q}_p}H^1_f(\mathbb{Q}_p,U[1])+\dim_{\mathbb{Q}_p}H^1_f(\mathbb{Q}_p,U[2])=\frac{1}{2}(3g-2)(g+1),
\end{align*}
and relies on the fact that $\mathrm{Fil}^0D_{\mathrm{dR}}(V_pJ)$ is $g$-dimensional. For the $3^{\mathrm{rd}}$ graded piece, the sequence \eqref{cubicDesc} implies that
\begin{align*}
\dim_{\mathbb{Q}_p}D_{\mathrm{dR}}(U[3])&=2g\bigg(\binom{2g}{2}-1\bigg)-\binom{2g}{3} \\
\dim_{\mathbb{Q}_p}\mathrm{Fil}^0D_{\mathrm{dR}}(U[3])&= g\binom{g}{2}-\binom{g}{3},
\end{align*}
from which the result follows.
\end{proof}

\subsection{Consequences assuming the Bloch--Kato conjectures} 

The following consequence of the Bloch--Kato conjectures \cite[Conjecture 5.3(i)]{bloch1990functions} facilitates the conditional finiteness criteria alluded to in the introduction. 

\begin{conjecture}[\texorpdfstring{\cite[Prediction 4.1]{bellBK}}{[Bel09, Prediction 4.1]}]\label{BKConjectures}
Let $K$ be a number field and let $V$ be a geometric $\mathbb{Q}_p$-representation of $G_K$ that is pure of non-negative weight. Then $H^1_f(K,V)=0$. 
\end{conjecture}

The precise relationship between Conjecture \ref{BKConjectures} and the Bloch--Kato conjectures is detailed in the discussion after Conjecture 2.8 in \cite{Cklocus}, for example. 

In the depth 2 case, under the assumption of Conjecture \ref{BKConjectures}, Balakrishnan--Dogra \cite[Lemma 2.6]{QC2} show  that $X(\mathbb{Q}_p)_2$ is finite provided the Mordell--Weil rank of $J$ is strictly less than $g^2+\rho-1$, where $\rho$ denotes the rank of the N\'{e}ron--Severi group of $J$. By a similar calculation, we obtain the following analogue in depth 3.

\begin{lemma}\label{CondDep3}
Assume that Conjecture \ref{BKConjectures} holds. Then the depth $3$ set $X(\mathbb{Q}_p)_3$ is finite whenever 
\begin{align}\label{dep3ineq}
\mathrm{rk}_{\mathbb{Z}} J(\mathbb{Q})<\frac{1}{3}(4g^3+3g^2-4g)+\rho-1.
\end{align}
\end{lemma}
\begin{proof}
If $W$ is a geometric $\mathbb{Q}_p$-representation of $G_K$ that is pure of negative weight, then \cite[Remark II.2.2.2]{FPR} implies that
\begin{align}\label{FPRSES}
\begin{split}
&\dim_{\mathbb{Q}_p}H^1_f(\mathbb{Q}_p,W)-\dim_{\mathbb{Q}_p}H^1_f(\mathbb{Q},W) \\ &= \dim_{\mathbb{Q}_p}H^0(\mathbb{R},W)-\dim_{\mathbb{Q}_p}H^1_f(\mathbb{Q},W^*(1))+\dim_{\mathbb{Q}_p}H^0(\mathbb{Q},W^*(1)).
\end{split}
\end{align}
Conjecture \ref{BKConjectures} implies that $H^1_f(\mathbb{Q},W^*(1))$ is trivial for both $W=U[2]$ and $W=U[3]$. Additionally, $H^0(\mathbb{Q},U[3]^*(1))$ is trivial for weight reasons, while 
\begin{align*}
\dim_{\mathbb{Q}_p}H^0(\mathbb{Q},U[2]^*(1))=\rho -1,
\end{align*}
by the proof of \cite[Lemma 2.5]{QC2}. Finally, the group $H^0(\mathbb{R},W)$ is the subspace of $W$ invariant under the action of complex conjugation, which, together with the isomorphism $V_pJ\simeq \mathbb{Q}_p^g\oplus \mathbb{Q}_p(1)^g$ of $G_{\mathbb{R}}$-modules, yields the following dimensions through simple calculation:
\begin{align*}
\dim_{\mathbb{Q}_p}H^0(\mathbb{R},U[2])=g(g-1), \tab \dim_{\mathbb{Q}_p}H^0(\mathbb{R},U[3])=\frac{4}{3}g(g^2-1). 
\end{align*}
Substituting these simplifications into \eqref{FPRSES}, while assuming \eqref{dep3ineq}, we have
\begin{align*}
\dim H^1_f(\mathbb{Q}_p,U_3)-\dim \mathrm{Sel}(U_3) \geq g-\mathrm{rk}_{\mathbb{Z}}J(\mathbb{Q})+g(g-1)+\rho-1+\frac{4}{3}g(g^2-1)>0,
\end{align*}
which implies finiteness of $X(\mathbb{Q}_p)_3$ by Theorem \ref{KimThm}.
\end{proof}

\begin{remark}\label{BK3remark}
From the proofs of Lemmas \ref{netancor3} and \ref{CondDep3}, it is possible to deduce the conditional dimension of the global Selmer group $H^1_f(\mathbb{Q},U[3])$ under the assumption of Conjecture \ref{BKConjectures}. Specifically, from \eqref{FPRSES}, we obtain 
\begin{align*}
\dim_{\mathbb{Q}_p}H^1_f(\mathbb{Q},U[3])=\dim_{\mathbb{Q}_p}H^1_f(\mathbb{Q}_p,U[3])-\dim_{\mathbb{Q}_p}H^0(\mathbb{R},U[3])=g(g^2-1). 
\end{align*}
\end{remark}

\subsection{2-descent for Bloch--Kato Selmer groups associated to hyperelliptic curves}

Henceforth, we focus on bounding the $\mathbb{Q}_p$-dimension of $H^1_f(\mathbb{Q},U[3])$. As discussed in the introduction, this is achieved by adapting the methods of \cite{bk1,bk2} to the cubic case. Therefore, we briefly recall the general strategy of these methods in obtaining quadratic Chabauty criteria, for which we require the following definition. 

\begin{definition}\label{h1fdef}
Let $K$ be a number field, let $V$ be a finite-dimensional $\mathbb{Q}_p$-vector space with a continuous action of $G_K$, and let $T\subseteq V$ be a $G_K$-stable $\mathbb{Z}_p$-lattice. For a prime $v$ of $K$, define $H^1_f(K_v,T)$ as the preimage of $H^1_f(K_v,V)$ under the natural map in cohomology induced by the inclusion $T\subseteq V$. We define the corresponding global analogue by 
\begin{align}\label{IntegralBK}
H^1_f(K,T)\coloneqq \bigcap_{v} \mathrm{ker}\big(H^1(K,T)\rightarrow H^1(K_v,T)/H^1_f(K_v,T)\big),
\end{align}
where the intersection varies over all primes $v$ of $K$. For a prime $v$ in $K$, define $H^1_f(K_v,T\otimes \mathbb{F}_p)$ as the image of $H^1_f(K_v,T)$ under the injection 
\begin{align*}
H^1(K_v,T)\otimes \mathbb{F}_p\hookrightarrow H^1(K_v,T\otimes \mathbb{F}_p),
\end{align*}
induced by multiplication by $p$ on $T$. The global analogue $H^1_f(K,T\otimes \mathbb{F}_p)$ is then defined analogously to \eqref{IntegralBK}. 
\end{definition}

Consequently, to bound the $\mathbb{Q}_p$-dimension of $H^1_f(K,V)$, it suffices to bound the $\mathbb{F}_p$-dimension of the group $H_f^1(K,T\otimes \mathbb{F}_p)$ associated to a Galois stable lattice $T\subseteq V$. In the depth 2 case, Dogra showed that the $\mathbb{Q}_p$-dimension of $H^1_f(\mathbb{Q},U[2])$ is bounded by the $\mathbb{F}_p$-dimension of $H^1_f(\mathbb{Q},{\wedge^2J[p]})$; this is itself contained in the classes that are unramified outside a suitably chosen set of primes. In the case that $X$ is a hyperelliptic curve and $p=2$, an approximation of $H^1_f(\mathbb{Q},{\wedge^2 J[2]})$ is effectively computable, in part due to the field-theoretic description of the 2-torsion subgroup $J[2]$ of a hyperelliptic Jacobian. 

We exploit this same setup in the cubic case. As such, in the remainder of the paper we consider only $p=2$ and, unless otherwise stated, we assume that $X$ is an odd hyperelliptic curve over a field $K$ of characteristic different from 2. Although we are primarily interested in curves over $\mathbb{Q}$, the following statements hold in this generality.

To describe an analogue of $H^1_f(\mathbb{Q},\wedge^2J[2])$ in the depth 3 case, we require the following basic lemma. Here, and throughout the remainder of this paper, if $f\in K[x]$ is a separable polynomial, then $\mathcal{R}_f$ will denote the $G_K$-set of roots of $f$. 

\begin{lemma}\label{wedgenlemma}
Let $X$ be a hyperelliptic curve over $K$ defined by the equation $y^2=f(x)$, where the degree of $f\in K[x]$ is odd. Then, for each integer $n>1$, there is a split short exact sequence 
\begin{align}\label{wedgenSES}
0\rightarrow {\wedge^n J[2]} \rightarrow {\wedge^n \mathbb{F}_2[\mathcal{R}_f]}&\rightarrow {\wedge^{n-1} J[2]} \rightarrow 0,
\end{align}
induced by the map 
\begin{align}\label{explicitwedgenmap}
\begin{split}
 {\wedge^n \mathbb{F}_2[\mathcal{R}_f]}&\rightarrow {\wedge^{n-1} J[2]} \\
\alpha_1\wedge \cdots \wedge \alpha_n &\mapsto \sum_{i=1}^n \alpha_1 \wedge\cdots \wedge \widehat{\alpha}_i \wedge \cdots \wedge \alpha_n,
\end{split}
\end{align}
where $\widehat{\alpha}_i$ denotes omission of the $i^{\mathrm{th}}$ factor. 
\end{lemma}
\begin{proof}
The existence of the sequence follows from the well-known split short exact sequence\footnote{Throughout this paper, $J[2]$ will be understood as the submodule of $\mathbb{F}_2[\mathcal{R}_f]$ defined by this exact sequence, and we will rarely require its description in terms of, say, divisors on $X$.}
\begin{align}\label{DefJ2}
0\rightarrow J[2] \rightarrow \mathbb{F}_2[\mathcal{R}_f] \xrightarrow{\mathrm{deg}}\mathbb{F}_2 \rightarrow 0,
\end{align}
which is contained in \cite[Eq. (5)]{bk1} and the discussion thereafter, for example. To see that it splits, note that if 
\begin{align*}
\varphi:\mathbb{F}_2[\mathcal{R}_f]\rightarrow J[2]
\end{align*}
defines a retraction of \eqref{DefJ2}, then $\wedge^n \varphi$ defines a retraction of \eqref{wedgenSES}.
\end{proof}

Using the exact sequence \eqref{wedgenSES} for $n=2$ and $n=3$, we define the following $G_K$-modules that arise in the description of the mod-2 quotient of a Galois stable lattice within $U[3]$. Firstly, we define $\overline{\wedge^2 J[2]}$ to be the cokernel of the map 
\begin{align}\label{nonsplitmapf2}
\begin{split}
\mathbb{F}_2 & \rightarrow \wedge^2J[2]\\
1&\mapsto \sum_{i<j} \alpha_i\wedge \alpha_j.
\end{split}
\end{align}
This notation is chosen in analogy with the notation $\overline{\wedge^2 V}\coloneqq U[2]$ appearing in quadratic Chabauty literature. Next, throughout this paper the quotient $J[2]\otimes \overline{\wedge^2 J[2]}/{\wedge^3 J[2]}$ will be understood as the cokernel of the map 
\begin{align}\label{wedge3map}
\wedge^3J[2]\rightarrow J[2]\otimes \overline{\wedge^2 J[2]},
\end{align}
induced by 
\begin{align}\label{reflate}
\begin{split}
{\wedge^3 \mathbb{F}_2[\mathcal{R}_f]}&\rightarrow \mathbb{F}_2[\mathcal{R}_f]\otimes {\wedge^2 \mathbb{F}_2}[\mathcal{R}_f]\\
\alpha_1\wedge\alpha_2 \wedge \alpha_3 &\mapsto \alpha_1 \otimes \big(\alpha_2\wedge \alpha_3\big)+\alpha_2 \otimes \big(\alpha_3\wedge \alpha_1\big)+\alpha_3 \otimes \big(\alpha_1\wedge \alpha_2\big).
\end{split}
\end{align}
With this notation, we obtain the following crude bound for the desired Bloch--Kato Selmer group.

\begin{proposition}\label{MainApprox}
Let $X$ be a hyperelliptic curve over $\mathbb{Q}$ defined by the equation $y^2=f(x)$, where the degree of $f\in \mathbb{Q}[x]$ is odd, and let $S$ be a finite set of primes that contains $\{2,\infty\}$ and the primes of bad reduction for $X$. Then
\begin{align}\label{approxinclusion}
\dim_{\mathbb{Q}_2}H^1_f\big(\mathbb{Q}, U[3]\big)< \dim_{\mathbb{F}_2} H^1(G_S, J[2]\otimes \overline{\wedge^2 J[2]}/{\wedge^3J[2]}).
\end{align}
\end{proposition}
\begin{proof}
To simplify notation, we denote the $\mathbb{Q}_2$-rational and $\mathbb{Z}_2$-integral Tate modules of $J$ by $V\coloneqq V_2J$ and $T \coloneqq T_2J$, respectively. The result follows straightforwardly once we define a Galois stable $\mathbb{Z}_2$-lattice within $U[3]$. 

To that end, first define $\overline{\wedge^2 T}$ as the cokernel of the map dual to the cup product 
\begin{align}\label{integralcup}
\cup:\wedge^2 H^1_{\textnormal{\'{e}t}}(X_{\overline{\mathbb{Q}}},\mathbb{Z}_2)\rightarrow H^2_{\textnormal{\'{e}t}}(X_{\overline{\mathbb{Q}}},\mathbb{Z}_2)\simeq \mathbb{Z}_2(-1).
\end{align}
To see that $\overline{\wedge^2 T}$ itself defines a $\mathbb{Z}_2$-lattice within $\overline{\wedge^2 V}=U[2]$, we note that $\overline{\wedge^2 T}$ admits the following explicit description: if $\{e_i,e_{g+i}\}_{i=1}^g$ is a basis of $T$ that is dual to the symplectic basis of $H^1_{\textnormal{\'{e}t}}(X_{\overline{\mathbb{Q}}},\mathbb{Z}_2)$ (with respect to the cup-product pairing), then $\overline{\wedge^2 T}$ is the cokernel of the map 
\begin{align*}
\mathbb{Z}_2(1) &\rightarrow \wedge^2T \\ 
1 &\mapsto \sum_{i=1}^ge_i\wedge e_{g+i},
\end{align*}
which follows immediately from the definition of a symplectic pairing. In particular, $\overline{\wedge^2 T}$ is torsion-free and the set 
\begin{align}\label{IntermediateBasis}
\{e_i\wedge e_j:1\leq i<j\leq 2g, \ (i,j)\neq (g,2g)\}
\end{align}
constitutes a $\mathbb{Z}_2$-basis. 

Consequently, a canonical candidate for a $G_K$-stable lattice within $U[3]$ is the cokernel $\mathcal{T}\coloneqq T\otimes\overline{\wedge^2 T}/{\wedge^3T}$ of the integral analogue
\begin{align}\label{Z2mapcube}
\begin{split}
{\wedge^3 T}&\rightarrow T\otimes{\wedge^2 T} \\
e_i\wedge e_j \wedge e_k &\mapsto e_i\otimes (e_j \wedge e_k)+ e_j \otimes (e_k\wedge e_i) + e_k \otimes (e_i\wedge e_j).
\end{split}
\end{align}
to the injection in \eqref{cubicDesc}. We need only prove that this cokernel is torsion-free. To this end, given that ${\wedge^3 T}\rightarrow T \otimes \overline{\wedge^2 T}$ is injective, it suffices to show that the $(\mathrm{rk}_{\mathbb{Z}_2}T\otimes \overline{\wedge^2T})\times(\mathrm{rk}_{\mathbb{Z}_2}{\wedge^3 T})$-matrix representing this map has a $(\mathrm{rk}_{\mathbb{Z}_2}{\wedge^3 T}) \times (\mathrm{rk}_{\mathbb{Z}_2}{\wedge^3 T}) $-minor of determinant $\pm1$. If this holds, then, by the theory of Smith Normal Forms, the product of all elementary divisors of the matrix is $1$, from which we conclude that its cokernel is torsion-free. This is straightforward to prove for this particular map, given that each column of the matrix contains a $1$ for which it is the only non-zero entry in its row. 

To this end, via the basis \eqref{IntermediateBasis} of $\overline{\wedge^2 T}$, we fix the following basis of $T\otimes \overline{\wedge^2T}$:
\begin{align*}
\{e_i\otimes \overline{e_j\wedge e_k}:1 \leq i \leq 2g,\ 1\leq j<k\leq 2g \text{ with } (j,k)\neq (g,2g)\}.
\end{align*}
Then the image of $e_i\wedge e_j \wedge e_k\in {\wedge^3 T}$, for $\{g,2g\}\not\subseteq\{i,j,k\}$, under the map \eqref{Z2mapcube} is 
\begin{align*}
e_i\otimes \overline{e_j\wedge e_k}+e_j \otimes \overline{e_k\wedge e_i}+e_k\otimes \overline{e_i\wedge e_j},
\end{align*}
while the image of $e_i\wedge e_g\wedge e_{2g}$ is 
\begin{align*}
e_g\otimes \overline{e_{2g}\wedge e_i}+e_{2g}\otimes \overline{e_i\wedge e_g}-\sum_{\ell =1}^{g-1}e_i\otimes \overline{e_\ell \wedge e_{\ell+g}}.
\end{align*}
Accordingly, with respect to these bases, the matrix representing \eqref{Z2mapcube} is of the desired form. Indeed, the summand $e_j\otimes \overline{e_{j+g}\wedge e_i}$ in the image of $e_i\wedge e_j \wedge e_{j+g}$ does not appear as the summand in the image of another basis element; for the remaining $e_i\wedge e_j \wedge e_k$ not of this form, every summand appears uniquely in its image. We conclude that $\mathcal{T}$ is torsion-free. 

To identify a description of the mod-2 quotient of this lattice, note that the mod-2 analogue of \eqref{integralcup} may be identified with the Weil pairing $\wedge^2 J[2] \rightarrow \mathbb{F}_2$ (see, for instance, \cite[Remark V.2.4(f)]{MilneOGEC}). Consequently, upon dualising we obtain that the mod-2 quotient of $\overline{\wedge^2 T}$ is isomorphic to $\overline{\wedge^2 J[2]}$. Finally, this implies that $J[2]\otimes {\overline{\wedge^2 J[2]}}/{\wedge^3 J[2]}$ is isomorphic to the mod-2 quotient of $T\otimes \overline{\wedge^2T}/{\wedge^3 T}$. 

In summary, the $\mathbb{Q}_2$-dimension of $H^1_f(\mathbb{Q},U[3])$ is bounded by the $\mathbb{F}_2$-dimension of the image of $H^1_f(\mathbb{Q},\mathcal{T})$ in $H^1(G_S,J[2]\otimes {\overline{\wedge^2 J[2]}}/{\wedge^3 J[2]})$. In particular, the cruder bound in \eqref{approxinclusion} holds. 
\end{proof}

The conditional applicability of cubic Chabauty in Lemma \ref{CondDep3} suggests that the rudimentary approximation afforded by Lemma \ref{MainApprox} should already suffice to deduce finiteness of $X(\mathbb{Q}_2)_3$ for examples beyond the quadratic Chabauty method. This is in contrast with the quadratic Chabauty descent papers, which each have to refine the analogous crude approximation to obtain appropriate examples. Accordingly, the subsequent section concerns effectively computing the $\mathbb{F}_2$-dimension of the cohomology group on the right-hand side of \eqref{approxinclusion}.

\begin{remark}\label{GSBK3}
The inequality \eqref{approxinclusion} from Proposition \ref{MainApprox} is strict, and Remark \ref{BK3remark} can be used to deduce the optimal bound. Specifically, \cite[Proposition 3.8]{bloch1990functions} states that $H^1_f(\mathbb{Q}_2,U[3])$ and $H^1_f(\mathbb{Q}_2,U[3]^*(1))$ are exact annihilators under the cup product pairing
\begin{align*}
H^1(\mathbb{Q}_2,U[3])\times H^1(\mathbb{Q}_2,U[3]^*(1))\xrightarrow{\cup} H^2(\mathbb{Q}_2,\mathbb{Q}_2(1))\simeq \mathbb{Q}_2, 
\end{align*}
which yields the following unconditional difference in dimensions:
\begin{align*}
\dim_{\mathbb{Q}_2}H^1(\mathbb{Q}_2,U[3])-\dim _{\mathbb{Q}_2}H^1_f(\mathbb{Q}_2,U[3])=\dim_{\mathbb{Q}_2}H^1_f(\mathbb{Q}_2,U[3]^*(1)).
\end{align*} 
Since $U[3]^*(1)$ is pure of non-zero weight, it follows that 
\begin{align*}
\dim_{\mathbb{Q}_2}H^1_f(\mathbb{Q}_2,U[3]^*(1))=\dim_{\mathbb{Q}_2}D_{\mathrm{dR}}(U[3]^*(1))/\mathrm{Fil}^0D_{\mathrm{dR}}(U[3]^*(1))=\dim_{\mathbb{Q}_2}\mathrm{Fil}^0D_{\mathrm{dR}}(U[3]),
\end{align*}
and this final dimension is computed in the proof of Lemma \ref{netancor3}. By the proof of \cite[Proposition II.2.2.1]{FPR}, Conjecture \ref{BKConjectures} implies that this local disparity lifts globally. Indeed, we obtain a surjection 
\begin{align*}
H^1(G_S,U[3]) \rightarrow \bigoplus_{p \in S}H^1(\mathbb{Q}_p,U[3])/H^1_f(\mathbb{Q}_p,U[3]),
\end{align*}
given that $H^1_f(\mathbb{Q},U[3]^*(1))=0$ under the assumption of the conjecture, where $S$ denotes a finite set of primes containing $\{2,\infty\}$ and the primes of bad reduction for $X$. Consequently, the sharpest bound for $\dim_{\mathbb{Q}_2}H^1_f(\mathbb{Q},U[3])$ obtainable by the methods of this paper is 
\begin{align*}
g(g^2-1)+g\binom{g}{2}-\binom{g}{3} = \frac{4}{3}g(g^2-1). 
\end{align*}
In particular, the sharpest bound we expect for a genus 2 curve is 8.
\end{remark}

In the following section, computing the dimension on the right-hand side of \eqref{approxinclusion} will amount to effectively determining the $\mathbb{F}_2$-dimensions of cohomology groups of the form $H^1(G_{K,S},\mathrm{Ind}_K^L\mathbb{F}_2)$, for $L$ a finite \'{e}tale $K$-algebra and $S$ a finite set of primes in $K$ containing those above $\{2,\infty\}$. To that end, note that we may first identify $H^1(G_{K,S},\mathrm{Ind}_K^L\mathbb{F}_2)$ with the Selmer group $H^1_S(K,\mathrm{Ind}_K^{K_f}\mathbb{F}_2)$ of elements in $H^1(K,\mathrm{Ind}_K^L\mathbb{F}_2)$ that are unramified outside of $S$. Consequently, we define $(L^\times \otimes \mathbb{F}_2)_S$ as the image of $H^1_S(K,\mathrm{Ind}_K^L\mathbb{F}_2)$ under the isomorphism 
\begin{align*}
H^1(K,\mathrm{Ind}_K^L\mathbb{F}_2)\xrightarrow{\sim} L^\times \otimes \mathbb{F}_2,
\end{align*}
which is induced by Shapiro's lemma and Kummer theory. 

\begin{lemma}[\texorpdfstring{\cite[Lemma 12.6]{poonsch}}{}]\label{Scohomolem}
Let $K$ be a number field, let $L$ be a finite \'{e}tale $K$-algebra, and let $S$ be a finite set of primes of $K$ that contains the primes above $2$ and the archimedean places. Then we have the short exact sequence 
\begin{align*}
0\rightarrow \mathcal{O}_{L,S}^\times \otimes \mathbb{F}_2 \rightarrow (L^\times \otimes \mathbb{F}_2)_S\rightarrow \mathrm{Cl}_S(L)[2] \rightarrow 0.
\end{align*}
\end{lemma}
Note that, while \cite[Proposition 12.6]{poonsch} is stated for $L=K_f$, its proof holds to the generality of this lemma. Here we are assuming that if we have the decomposition $L=\prod L_i$, then the ring of $S$-integers and the $S$-class group of the \'{e}tale algebra $L$ are defined as 
\begin{align*}
\mathcal{O}_{L,S}\coloneqq \prod_i\mathcal{O}_{L_i,S}, \tab \mathrm{Cl}_S(L)\coloneqq \prod_i\mathrm{Cl}_S(L_i),
\end{align*}
where $\mathcal{O}_{L_i,S}$ (resp. $\mathrm{Cl}_S(L_i)$) denotes the $S$-integers (resp. $S$-class group) of the field $L_i$.

\section{Bounds in depth 3}\label{MainCohomSection}

Although this section concerns obtaining the $\mathbb{F}_2$-dimension of $H^1(G_S,J[2]\otimes \overline{\wedge^2 J[2]}/{\wedge^3 J[2]})$ to exploit Proposition \ref{MainApprox}, the required preliminary results in Galois cohomology hold for more general fields. Accordingly, we continue to assume that $K$ is a field of characteristic different from $2$ and $X$ is an odd hyperelliptic curve over $K$. We reiterate that the assumption on the degree of the defining polynomial of $X$ is essential, given that most results in this section do not necessarily hold in the even degree case. 

\begin{lemma}\label{wedge3split}
The short exact sequence
\begin{align}\label{FinallySplit}
0\rightarrow \wedge^3J[2] \rightarrow J[2]\otimes \overline{\wedge^2 J[2]}\rightarrow J[2]\otimes \overline{\wedge^2 J[2]}/{\wedge^3 J[2]} \rightarrow 0,
\end{align}
induced by \eqref{wedge3map}, is split. 
\end{lemma}
\begin{proof}
To show that the sequence is split, it suffices to construct a map
\begin{align*}
\mathbb{F}_2[\mathcal{R}_f]\otimes \wedge^2 \mathbb{F}_2[\mathcal{R}_f]\rightarrow {\wedge^3 \mathbb{F}_2[\mathcal{R}_f]}
\end{align*}
that both defines a retraction of 
\begin{align}\label{retractionmap}
{\wedge^3J[2]}\rightarrow J[2]\otimes {\wedge^2 J[2]}
\end{align}
and vanishes on the image of $J[2]$ under the map induced by \eqref{nonsplitmapf2}. To that end, for distinct $i,j,k$, consider the element
\begin{align*}
D_{i,j,k}\coloneqq \sum_{m,n}\alpha_i\wedge \alpha_m\wedge \alpha_n + \sum_{r,s,t \notin \{i,j,k\}}\alpha_r\wedge \alpha_s\wedge\alpha_t \in {\wedge^3 \mathbb{F}_2[\mathcal{R}_f]},
\end{align*}
from which we define the map 
\begin{align*}
\varphi: \mathbb{F}_2[\mathcal{R}_f]\otimes {\wedge^2 \mathbb{F}_2[\mathcal{R}_f]}&\rightarrow {\wedge^3\mathbb{F}_2[\mathcal{R}_f]}\\
\alpha_i\otimes (\alpha_j \wedge \alpha_k) &\mapsto \begin{cases}
D_{i,j,k} &\text{if $i,j,k$ are distinct;}\\
0 &\text{otherwise.}
\end{cases}
\end{align*}
Without loss of generality, to describe the composition of \eqref{reflate} with $\varphi$, we consider the image of the single triple $\alpha_1\wedge \alpha_2 \wedge \alpha_3 \in {\wedge^3 \mathbb{F}_2[\mathcal{R}_f]}$; rather than working directly with the resulting sum, it is enough to count the number of times a given $\alpha_i\wedge \alpha_j \wedge \alpha_k$ appears as a summand: 
\begin{itemize}
    \item If $\{i,j,k\}\cap \{1,2,3\}=\emptyset$, then $\alpha_i\wedge \alpha_j\wedge \alpha_k$ appears three times. 
    \item If $\# (\{i,j,k\}\cap \{1,2,3\})=1$, then $\alpha_i\wedge \alpha_j\wedge \alpha_k$ appears once. 
    \item If $\# (\{i,j,k\}\cap \{1,2,3\})=2$, then $\alpha_i\wedge \alpha_j\wedge \alpha_k$ appears twice.
    \item The triple $\alpha_1\wedge \alpha_2 \wedge \alpha_3$ appears 3 times.
\end{itemize}
In summary, we obtain that the above composition defines the map 
\begin{align}\label{anotheralmostretract}
\begin{split}
{\wedge^3\mathbb{F}_2[\mathcal{R}_f]}&\rightarrow {\wedge^3\mathbb{F}_2[\mathcal{R}_f]}\\
\alpha_1\wedge \alpha_2 \wedge \alpha_3 &\mapsto \alpha_1\wedge \alpha_2 \wedge \alpha_3 +\sum_{\substack{i<j<k \\ |\{i,j,k\}\cap \{1,2,3\}|\leq 1}} \alpha_{i}\wedge \alpha_j\wedge \alpha_k. 
\end{split}
\end{align}
To show that this restricts to identity on ${\wedge^3 J[2]}$, first note that the $n=3$ case of \eqref{wedgenSES} implies that ${\wedge^3 J[2]}$ is the $\mathbb{F}_2$-span of elements of the form 
\begin{align*}
\alpha_{i_1}\wedge \alpha_{i_2} \wedge \alpha_{i_3} +\alpha_{i_1}\wedge \alpha_{i_2} \wedge \alpha_{i_4} +\alpha_{i_1} \wedge \alpha_{i_3} \wedge \alpha_{i_4} +\alpha_{i_2} \wedge \alpha_{i_3}\wedge\alpha_{i_4}.
\end{align*}
Through another count of the frequency a given $\alpha_i\wedge \alpha_j \wedge \alpha_k$ appears in the image of this $\mathbb{F}_2$-basis element under \eqref{anotheralmostretract}, we obtain that this composition is indeed identity on ${\wedge^3 J[2]}$. We conclude that $\varphi$ defines the desired retraction of \eqref{retractionmap}. 

It remains to show that $\varphi$ vanishes on the image of $J[2]$ under the map 
\begin{align}\label{imageJ2map}
\begin{split}
\mathbb{F}_2[\mathcal{R}_f]&\rightarrow \mathbb{F}_2[\mathcal{R}_f] \otimes {\wedge^2 J[2]}\\
\alpha &\mapsto \sum_{i<j}\alpha\otimes (\alpha_i\wedge \alpha_j).
\end{split}
\end{align}
First note that the element $\alpha_1\in \mathbb{F}_2[\mathcal{R}_f]$ is mapped under the composition of \eqref{imageJ2map} with $\varphi$ to the following sum:
\begin{align}\label{awfsum}
\binom{\deg f -1}{2}\bigg(\sum_{m<n}\alpha_1\wedge\alpha_m\wedge \alpha_n\bigg)+\sum_{\substack{i<j\\i,j\neq 1}}\sum_{r,s,t \notin \{1,i,j\} } \alpha_r\wedge\alpha_s\wedge \alpha_t.
\end{align}
A given $\alpha_r\wedge \alpha_s\wedge \alpha_t$ (for $r,s,t \neq 1$) appears in this right-hand sum for any $i,j \notin \{1,r,s,t\}$, and hence \eqref{awfsum} simplifies to
\begin{align*}
&\binom{\deg f -1}{2}\bigg(\sum_{m<n}\alpha_1\wedge\alpha_m\wedge \alpha_n\bigg)+ \binom{\deg f -4}{2}\bigg(\sum_{r,s,t \neq 1}\alpha_r\wedge \alpha_s \wedge \alpha_t\bigg)\\
&=\binom{\deg f -4}{2} \bigg(\sum_{r,s,t}\alpha_r\wedge \alpha_s\wedge \alpha_t\bigg),
\end{align*}
where this final equality follows as the two binomial coefficients have the same parity whenever $\deg f$ is odd. Consequently, $\varphi$ vanishes on the image of even degree sums of roots of $f$ under \eqref{imageJ2map}, as required.
\end{proof}

Consequently, to compute the dimension of the right-hand side of \eqref{approxinclusion}, it suffices to compute the dimensions of $H^1(G_S,{\wedge^3 J[2]})$ and $H^1(G_S,J[2]\otimes \overline{\wedge^2 J [2]})$ independently. Towards either calculation, we require some notation: for each integer $1\leq n\leq \deg f$, denote by $K_{f,n}$ the following \'{e}tale algebra:
\begin{align*}
K_{f,n}\coloneqq K[t_1,\dots,t_n][\tfrac{1}{t_i-t_j}:i\neq j]/(f(t_1),\dots,f(t_n)).
\end{align*}
If $0\leq k\leq n$, then 
\begin{align*}
K_f^{(n-k,k)}\subseteq K_{f,n}
\end{align*} 
will denote the subalgebra fixed by the action of the permutation group $S_{n-k}$ on $\{t_1,\dots,t_{n-k}\}$. For simplicity, set $K_{f}^{(n)}\coloneqq K_{f}^{(n,0)}$. In particular, \cite[Proposition 1]{bk1} shows that 
\begin{align}\label{WedgeSqDesc}
\begin{split}
H^1(K,\wedge^2J[2])&\simeq \mathrm{ker}\big(H^1(K,\wedge^2 \mathbb{F}_2[\mathcal{R}_f])\rightarrow H^1(K,\mathbb{F}_2[\mathcal{R}_f])\big)\\
&\simeq \mathrm{ker}\big(K_f^{(2),\times}\otimes \mathbb{F}_2 \rightarrow K_f^\times \otimes \mathbb{F}_2\big),
\end{split}
\end{align}
where the final map factors through the norm map associated to the extension of \'{e}tale algebras $K_{f,2}/K_f$, and $\mathcal{R}_f$ continues to denote the set of roots of $f$. In addition, these \'{e}tale algebras describe cohomology groups associated to the higher exterior powers as detailed in the following lemma.

\begin{lemma}\label{cubInd}
For $1\leq n \leq \deg f$, we have the isomorphism 
\begin{align*}
H^1(K,\wedge^n \mathbb{F}_2[\mathcal{R}_f])\simeq K_f^{(n),\times}\otimes\mathbb{F}_2.
\end{align*}
\end{lemma}
\begin{proof}
Suppose $\mathcal{R}\simeq \mathcal{R}_f$ is the $G_K$-set $\mathrm{Hom}_K(K_f,K^\mathrm{sep})$, and denote by $\mathcal{R}^{(n)}$ the set of unordered $n$-tuples of distinct elements in $\mathcal{R}$. Then we have the following isomorphism of $G_K$-modules:
\begin{align*}
\wedge^n \mathbb{F}_2[\mathcal{R}]&\xrightarrow{\sim}\mathbb{F}_2[\mathcal{R}^{(n)}]\\
s_1\wedge \dots \wedge s_n &\mapsto \{s_1,\dots,s_n\}.
\end{align*}
Consequently, by \cite[Lemma 3]{bk1}, the result follows if we establish the isomorphism of $G_K$-sets
\begin{align}\label{s(3)isom}
{\mathcal{R}^{(n)}}\simeq \mathrm{Hom}_K(K_f^{(n)},K^{\mathrm{sep}}).
\end{align}
This can be shown by an identical argument to the proof of \cite[Lemma 2]{bk1}: we include it here for the sake of completeness. 

If we denote by $\mathcal{R}_{(n)}$ the ordered $n$-tuples of distinct elements of $\mathcal{R}$, then the injection of $K_{f,n}$ into $K_f^{{\otimes n}}$ yields the following isomorphism of $G_K$-sets:
\begin{align}\label{s3equiv}
\begin{split}
\mathcal{R}_{(n)}&\xrightarrow{\sim} \mathrm{Hom}_K(K_{f,n},K^{\mathrm{sep}})\\
(s_1,\dots,s_n)&\mapsto s_1\otimes \dots \otimes s_n.
\end{split}
\end{align}
Clearly, there is a canonical action of $S_n$ on both sides of \eqref{s3equiv}, under which this isomorphism is equivariant. Since the quotient of $\mathcal{R}_{(n)}$ by this action is isomorphic to $\mathcal{R}^{(n)}$, we establish \eqref{s(3)isom} if the map 
\begin{align}\label{pullbackmap}
\mathrm{Hom}_K(K_{f,n},K^{\mathrm{sep}})/S_n\rightarrow \mathrm{Hom}_K(K_f^{(n)},K^{\mathrm{sep}}),
\end{align}
induced by pulling back the inclusion $K_f^{(n)}\hookrightarrow K_{f,n}$, is an isomorphism. Finally, this follows as the pullback is surjective and factors through the quotient by $S_n$, and the orders of the domain and codomain of \eqref{pullbackmap} coincide.
\end{proof}

As an immediate consequence of Lemma \ref{cubInd}, it is possible to obtain the dimension of $H^1(G_S,{\wedge^3 J[2]})$. Indeed, the $n=3$ case of Lemma \ref{wedgenlemma} implies that 
\begin{align}\label{wedge3J2dim}
\dim_{\mathbb{F}_2}H^1(G_S,\wedge^3 J[2])=\dim_{\mathbb{F}_2}H^1(G_S,\wedge^3 \mathbb{F}_2[\mathcal{R}_f])- \dim_{\mathbb{F}_2}H^1(G_S,\wedge^2 J[2]).
\end{align}
The dimensions on the right-hand side of this equality are effectively computable by Lemma \ref{cubInd}, the description \eqref{WedgeSqDesc} and Lemma \ref{Scohomolem}. 

It remains to obtain methods to effectively compute the dimension of $H^1(G_S,J[2]\otimes \overline{\wedge^2 J[2]})$; this is the object of the following two lemmas. 

\begin{lemma}\label{simple21isom}
We have the isomorphism 
\begin{align*}
H^1(K,\mathbb{F}_2[\mathcal{R}_f]\otimes {\wedge^2 J[2]})\simeq K_f^{(1,2),\times}\otimes \mathbb{F}_2.
\end{align*}
\end{lemma}
\begin{proof}
This is a consequence of the following basic result: if $G$ is a group, $H\subseteq G$ is a subgroup, $M$ is an $H$-module and $N$ is a $G$-module, then $(\mathrm{Ind}^G_HM)\otimes N \simeq \mathrm{Ind}^G_H(M\otimes \mathrm{Res}^G_HN)$ as $G$-modules. Applying this to the case of interest yields
\begin{align*}
\mathbb{F}_2[\mathcal{R}_f]\otimes {\wedge^2 J[2]}\simeq \mathrm{Ind}_K^{K_f}\mathbb{F}_2 \otimes \wedge^2 J[2] \simeq \mathrm{Ind}_K^{K_f}\mathrm{Res}_K^{K_f}{\wedge^2 J[2]}\simeq \mathrm{Ind}_K^{K_f}( \wedge^2 \mathrm{Ind}_{K_f}^{K_{f,2}}\mathbb{F}_2).
\end{align*}
The lemma then follows from the isomorphism 
\begin{align*}
\wedge^2 \mathrm{Ind}_{K_f}^{K_{f,2}}\mathbb{F}_2\simeq \mathrm{Ind}_{K_f}^{K_f^{(1,2)}}\mathbb{F}_2,
\end{align*}
which is implied by \cite[Lemmas 1 and 2]{bk1} and the definition of $K_f^{(1,2)}$. 
\end{proof}

\begin{lemma}\label{splitlemmaJ2}
The short exact sequence 
\begin{align}\label{biggernonsplit}
0\rightarrow J[2] \rightarrow J[2] \otimes \wedge^2 J[2] \rightarrow J[2] \otimes \overline{\wedge^2 J[2]} \rightarrow 0,  
\end{align}
induced by \eqref{nonsplitmapf2}, is split. 
\end{lemma}
\begin{proof}
If $\deg f\equiv 3 \textrm{ (mod $4$)}$, then the map 
\begin{align*}
\mathrm{deg}:{\wedge^2 \mathbb{F}_2}[\mathcal{R}_f]\rightarrow \mathbb{F}_2,\tab \alpha_i\wedge \alpha_j &\mapsto 1
\end{align*}
induces a retraction of \eqref{nonsplitmapf2}, given that $\binom{\deg f}{2}$ is odd. Consequently, \eqref{biggernonsplit} is split in this case.

Hence, assume that $\deg f \equiv 1 \textrm{ (mod $4$)}$. A retraction of \eqref{biggernonsplit} is induced by the following map:
\begin{align}\label{retractionJ2}
\begin{split}
\mathbb{F}_2[\mathcal{R}_f] \otimes {\wedge^2 \mathbb{F}_2[\mathcal{R}_f]}&\rightarrow \mathbb{F}_2[\mathcal{R}_f]\\
\alpha_i\otimes (\alpha_j\wedge \alpha_k)&\mapsto \begin{cases}
\alpha_i+\alpha_j+\alpha_k &\text{if $i,j,k$ are distinct;}\\
0 & \text{otherwise.}
\end{cases}
\end{split}
\end{align}
The image of $\alpha_1\in \mathbb{F}_2[\mathcal{R}_f]$ under the composition of \eqref{imageJ2map} with \eqref{retractionJ2} is the sum 
\begin{align*}
\sum_{\substack{i<j\\i,j\neq 1}}\alpha_1+\alpha_i+\alpha_j = \alpha_1+\sum_i\alpha_i,
\end{align*}
given that $\binom{\deg f-1}{2}$ is even in this case. Consequently, this composition is identity on the even degree sums of roots of $f$, as desired. 
\end{proof}

As a consequence of Lemma \ref{splitlemmaJ2}, we obtain the following equality of dimensions: 
\begin{align*}
\dim_{\mathbb{F}_2}H^1(G_S,J[2]\otimes \overline{\wedge^2 J[2]})&=\dim_{\mathbb{F}_2}H^1(G_S,J[2]\otimes{\wedge^2 J[2]})-\dim_{\mathbb{F}_2}H^1(G_S,J[2]).
\end{align*}
This, in turn, is precisely 
\begin{align}\label{dimforthm}
\dim_{\mathbb{F}_2}H^1(G_S,\mathbb{F}_2[\mathcal{R}_f]\otimes{\wedge^2 J[2]})-\dim_{\mathbb{F}_2}H^1(G_S,\wedge^2J[2])-\dim_{\mathbb{F}_2}H^1(G_S,J[2]),
\end{align}
given the split short exact sequence 
\begin{align}\label{J2otimeswedgej2SES}
0\rightarrow J[2] \otimes {\wedge^2 J[2]} \rightarrow \mathbb{F}_2[\mathcal{R}_f] \otimes {\wedge^2 J[2]}\rightarrow \wedge^2 J[2] \rightarrow 0,
\end{align}
induced by \eqref{DefJ2}. 

\begin{remark}
Although the $\mathbb{F}_2$-dimensions of $H^1(G_S,{\wedge^3J[2]})$ and $H^1(G_S,J[2]\otimes\overline{\wedge^2 J[2]})$ alone suffice for the purpose of computing the dimension of $H^1(G_S,J[2]\otimes \overline{\wedge^2 J[2]}/{\wedge^3J[2]})$, each of these cohomology groups admit a field-theoretic description in terms of the finite \'{e}tale algebras above; we collect them in this remark. 
\begin{enumerate}
    \item Consider the following diagram: 
    \begin{center}
    \begin{tikzcd}
    H^1(K,\wedge^3 \mathbb{F}_2[\mathcal{R}_f])\arrow{r}\arrow{d} & H^1(K,\mathbb{F}_2[\mathcal{R}_f]\otimes \wedge^2 \mathbb{F}_2[\mathcal{R}_f])\arrow{r} \arrow{d}& H^1(K,\wedge^2 \mathbb{F}_2[\mathcal{R}_f])\arrow{d}\\
    K_f^{(3),\times}\otimes \mathbb{F}_2 \arrow{r} & K_{f,2}^\times \otimes \mathbb{F}_2 \oplus K_f^{(1,2),\times}\otimes \mathbb{F}_2 \arrow{r} & K_f^{(2),\times} \otimes \mathbb{F}_2,
    \end{tikzcd}
    \end{center}
    where each of the maps are as follows:
    \begin{itemize}
        \item the vertical maps are the isomorphisms in Lemma \ref{cubInd}, Lemma \ref{simple21isom} and \eqref{WedgeSqDesc}, respectively;
        \item the first map in the bottom row is induced by the inclusion $K_f^{(3)}\hookrightarrow K_f^{(1,2)}$, and the second is the norm map; and
        \item the first map in the top row is induced by \eqref{reflate}, while the second is obtained from the map 
        \begin{align*}
        \mathbb{F}_2[\mathcal{R}_f] \otimes\wedge^2  \mathbb{F}_2[\mathcal{R}_f]
        &\longrightarrow \wedge^2 \mathbb{F}_2[\mathcal{R}_f] \\
        \alpha_1 \otimes  (\alpha_2 \wedge \alpha_3) 
        &\longmapsto  \alpha_2\wedge \alpha_3.
        \end{align*} 
    \end{itemize}
    Analogous to the proof of \cite[Lemma 6]{bk1}, commutativity of the left-hand (resp. right-hand) square above follows from Lemma 4 (resp. Lemma 5) from \cite{bk1}. This commutative diagram, coupled with the $n=3$ case of Lemma \ref{wedgenlemma}, implies that
    \begin{align*}
    H^1(K,\wedge^3J[2])\simeq \mathrm{ker}\big(K_f^{(3),\times}\otimes \mathbb{F}_2 \rightarrow K_f^{(2),\times}\otimes \mathbb{F}_2\big),
    \end{align*}
    where the map on the right-hand side is the composition 
    \begin{align*}
    K_f^{(3),\times}\otimes \mathbb{F}_2 \rightarrow K_f^{(1,2),\times}\otimes \mathbb{F}_2 \xrightarrow{\mathrm{Nm}}  K_f^{(2),\times}\otimes \mathbb{F}_2,
    \end{align*}
    induced by the inclusion of $K_f^{(3)}$ into $K_f^{(1,2)}$.

    A suitable modification of this argument implies the more general statement that $H^1(K,{\wedge^n J[2]})$ is isomorphic to the kernel of the composite map
    \begin{align*}
    K_f^{(n),\times}\otimes \mathbb{F}_2\rightarrow K_f^{(1,n-1),\times}\otimes \mathbb{F}_2\xrightarrow{\mathrm{Nm}} K_f^{(n-1),\times}\otimes \mathbb{F}_2,
    \end{align*}
    for $n>3$. This could be advantageous, for instance, in putative methods to develop analogous explicit descent techniques for higher graded pieces of the filtration on the fundamental group, as dictated by their descriptions in \cite[Corollary 8.3]{Hain1995InfinitesimalPO}.
    \item Exploiting the same lemmas \cite[Lemmas 4 and 5]{bk1} yields a field-theoretic description of $H^1(K,J[2]\otimes\overline{\wedge^2 J[2]})$. Indeed, the sequence \eqref{J2otimeswedgej2SES} implies that we have the isomorphism 
    \begin{align*}
    H^1(K,J[2]\otimes \wedge^2J[2])\simeq \mathrm{ker}\big(K_f^{(1,2),\times}\otimes \mathbb{F}_2 \xrightarrow{\mathrm{Nm}}K_f^{(2),\times}\otimes \mathbb{F}_2\big),
    \end{align*}
    while the cokernel of the map 
    \begin{align*}
    H^1(K,\mathbb{F}_2[\mathcal{R}_f])\rightarrow H^1(K,\mathbb{F}_2[\mathcal{R}_f]\otimes {\wedge^2 J[2]}),
    \end{align*}
    associated to \eqref{imageJ2map}, is induced by the inclusion $K_f\subseteq K_f^{(1,2)}$ of \'{e}tale algebras. 
\end{enumerate}
\end{remark}

The main result now follows immediately from the theory presented thus far. Prior to its statement, however, we define the following constant:
\begin{align}\label{dXSdef}
d_{X,S}\coloneqq \dim_{\mathbb{F}_2}H^1(G_S,\mathbb{F}_2[\mathcal{R}_f]\otimes {\wedge^2 J[2]})-\sum_{n=1}^3\dim_{\mathbb{F}_2}H^1(G_S,{\wedge^nJ[2]}),
\end{align}
where $H^1(G_S,\wedge^nJ[2])$ for $n=1$ is understood as $H^1(G_S,J[2])$. We will use the notation $d_X\coloneqq d_{X,S}$ whenever the set of primes $S$ is unambiguous.

\begin{theorem}\label{mainthm}
Let $X$ be a hyperelliptic curve over $\mathbb{Q}$ defined by the equation $y^2=f(x)$, where the degree of $f\in \mathbb{Q}[x]$ is odd. Let $S$ be a finite set of primes containing $\{2,\infty\}$ and the primes of bad reduction for $X$. Additionally, suppose that the following inequality holds:
\begin{align*}
\dim_{\mathbb{Q}_2}H^1_f(\mathbb{Q},U[2])+d_{X,S}<\frac{1}{6}(g+1)(2g+1)(7g-6)-\mathrm{rk}_{\mathbb{Z}}J(\mathbb{Q}).
\end{align*}
Then the depth $3$ Chabauty--Kim set $X(\mathbb{Q}_2)_3$ is finite. 
\end{theorem}
\begin{proof}
By Lemma \ref{wedge3split}, Lemma \ref{splitlemmaJ2} and the dimension calculation \eqref{dimforthm}, it follows that $d_{X,S}$ is precisely the $\mathbb{F}_2$-dimension of $H^1(G_S,J[2]\otimes \overline{\wedge^2 J[2]}/{\wedge^3 J[2]})$, which, given the choice of $S$, is a bound for the $\mathbb{Q}_2$-dimension of $H^1_f(\mathbb{Q},U[3])$ by Proposition \ref{MainApprox}. The result then follows from Lemma \ref{netancor3}.
\end{proof}

We reiterate that the methods of \cite{bk1,bk2,1stpaper} can be used to obtain a computable upper bound for the $\mathbb{Q}_2$-dimension of the Bloch--Kato Selmer group $H^1_f(\mathbb{Q},U[2])$. Additionally, Lemma \ref{Scohomolem} ensures that $d_{X,S}$ is itself effectively computable. Specifically, each summand in the definition of $d_{X,S}$ can be computed as follows:
\begin{itemize}
    \item The dimension of $H^1(G_S,\mathbb{F}_2[\mathcal{R}_f] \otimes {\wedge^2 J[2]})$ can be computed from Lemma \ref{simple21isom}.
    \item The dimension of $H^1(G_S,{\wedge^3 J[2]})$ can be computed using Lemma \ref{cubInd}, via the simplification \eqref{wedge3J2dim}.
    \item The dimensions of $H^1(G_S,\wedge^2 J[2])$ and $H^1(G_S,J[2])$ follow from known results (the isomorphism \eqref{WedgeSqDesc} and the sequence \eqref{DefJ2}). 
\end{itemize}
In the final section, we demonstrate $d_{X,S}$ is not only effectively computable, but computable in practice for low genus curves. 

\section{A genus 2 example}

For a given curve $X$ and a set of primes $S$, the primary computational obstruction to determining $d_{X,S}$ is computing the $S$-class group of $\mathbb{Q}_f^{(1,2)}$, an \'{e}tale algebra of degree $g(4g^2-1)$. Accordingly, in this section we focus on examples of genus 2 curves that satisfy the main finiteness criterion. Fortunately, there are straightforward simplifications in the genus 2 case, which we collect in the first subsection, and these will aid in the computation of $d_{X,S}$.

All computations in the remainder of this paper were implemented in \texttt{Magma} \cite{Magma}.

\subsection{Simplifications in genus 2}

For simplicity, assume that $f$ and $K$ are such that $K_f^{(1,2)}$ is a field. This will not be restrictive for our purposes as the example in the following subsection satisfies this for $K=\mathbb{Q}$. 

The first simplification arises as $J[2]$ is 4-dimensional in the genus 2 case, from which we obtain the isomorphism ${\wedge^3 J[2]}\simeq J[2]$; in particular, the dimensions of $H^1(G_S,{\wedge^3 J[2]})$ and $H^1(G_S,J[2])$ coincide. A further simplification comes from the fact that, in the genus 2 case, the field $K_f^{(1,2)}$ is a quadratic extension. To elaborate, let $g\in K_f[x]$ be the quartic polynomial defined by the factorisation $f(x)=(x-\alpha_1)g(x)$ over $K_f\simeq K(\alpha_1)$, and denote by $h\in K_f[x]$ the cubic resolvent of $g$. Let $K_{f,h}\coloneqq (K_f)_h$ denote the resulting degree 15 extension, which is contained in $K_f^{(1,2)}$ given that the roots of $h$ are of the form $\alpha_2\alpha_3+\alpha_4\alpha_5$ for distinct roots $\alpha_i$ of $g$. Finally, denote by $c\in K_{f,h}^\times\otimes \mathbb{F}_2$ the class for which $K_f^{(1,2)}\simeq K_{f,h}(\sqrt{c})$. 

The following basic lemma outlines the role of $K_{f,h}$ in circumventing the calculation of the class group of $K_f^{(1,2)}$:

\begin{lemma}\label{d1lemma}
Let $S$ be a finite set of primes in $K$ containing primes above $\{2,\infty\}$. Then we have the following inequality:
\begin{align}\label{lem21ineq}
\dim_{\mathbb{F}_2}H^1(G_{K,S},\mathbb{F}_2[\mathcal{R}_f]\otimes \wedge^2 J[2])\leq 2\big(\mathrm{rk}_{\mathbb{Z}}\mathcal{O}_{K_{f,h},S}^\times + \dim_{\mathbb{F}_2}\mathrm{Cl}_S(K_{f,h})[2]\big)+1.
\end{align}
\end{lemma}
\begin{proof}
This is a consequence of the following short exact sequence of $G_{K_{f,h}}$-modules: 
\begin{align}\label{quadboundary}
0\rightarrow \mathbb{F}_2 \rightarrow \mathrm{Ind}_{K_{f,h}}^{K_f^{(1,2)}}\mathbb{F}_2 \rightarrow \mathbb{F}_2 \rightarrow 0.
\end{align}
Specifically, by Lemma \ref{simple21isom} the dimension on the left-hand side of \eqref{lem21ineq} coincides with the dimension of the middle term in the exact sequence
\begin{align}\label{cokerkerSES}
0\rightarrow \mathrm{coker}(\delta_0)\rightarrow H^1(G_{K_{f,h},S},\mathrm{Ind}_{K_{f,h}}^{K_f^{(1,2)}}\mathbb{F}_2) \rightarrow \mathrm{ker}(\delta_1) \rightarrow 0
\end{align}
where $\delta_0$ and $\delta_1$ are the boundary maps
\begin{align*}
\delta_i&:H^i(G_{K_{f,h},S},\mathbb{F}_2)\rightarrow H^{i+1}(G_{K_{f,h},S},\mathbb{F}_2),
\end{align*}
associated to \eqref{quadboundary}. The boundary map $\delta_0$ is clearly injective as it sends $1\in \mathbb{F}_2$ to the class of $c\in K_{f,h}^\times \otimes \mathbb{F}_2$, so the dimension of its cokernel is $\dim_{\mathbb{F}_2}H^1(G_{K,S},\mathrm{Ind}_K^{K_{f,h}}\mathbb{F}_2)-1$. The result then follows from \eqref{cokerkerSES} and Lemma \ref{Scohomolem}.
\end{proof}

\begin{remark}
While Lemma \ref{d1lemma} suffices to give bounds sharp enough for the upcoming example, the crude bound on the kernel of $\delta_1$ used in its proof can be refined, since $\delta_1$ itself admits an explicit description: it sends $d \in H^1(G_{K_{f,h},S},\mathbb{F}_2)$ to the cup product $c\cup d $ in $H^2(G_{K_{f,h},S},\mathbb{F}_2)$. To obtain a field-theoretic description of this cup product is then possible via \cite{McCSha}.
\end{remark}

\subsection{A rank 4 curve} 

In this final subsection, we include an example of a hyperelliptic curve that satisfies the criterion of Theorem \ref{mainthm}, which is outside the scope of the quadratic Chabauty method. Indeed, this genus 2 curve has a Jacobian of Mordell--Weil rank 4 and N\'{e}ron--Severi rank 1. This latter condition follows as we can use \texttt{Magma} to confirm that the curve has analytic Jacobian with a trivial endomorphism algebra.  

Since no such odd hyperelliptic curve appears on the LMFDB \cite{lmfdb}, the following example was obtained by searching among defining polynomials with small coefficients. To ensure this search was computationally feasible, we first collected curves that satisfied the main criterion under the assumptions that $\dim_{\mathbb{Q}_2}H^1_f(\mathbb{Q},U[2])=3$ (the optimal bound achievable with the quadratic methods) and that $\mathrm{Cl}_S(\mathbb{Q}_f^{(1,2)})$ is 2-torsion-free, and then completed the computations on the resulting list of curves.

The following theorem details one such high rank example found in this search: 

\begin{theorem}\label{rank4thm}
Suppose that $X$ is the genus $2$ hyperelliptic curve over $\mathbb{Q}$, defined by the equation
\begin{align*}
y^2=4x^5 + 8x^4 + 4x^2 + 8x + 1,
\end{align*}
whose Jacobian has Mordell--Weil rank $4$. Then the depth $3$ Chabauty--Kim set $X(\mathbb{Q}_2)_3$ is finite. 
\end{theorem}
\begin{proof}
The curve has bad reduction only at $1115467$ and so we set $S=\{2,1115467,\infty\}$. To compute (a bound for) $d_X$, we use the splittings from Lemma \ref{wedgenlemma} and the fact that $J[2] \simeq \wedge^3 J[2]$ in the genus $2$ case to obtain the simplification
\begin{align}\label{dxsimp}
\begin{split}
d_{X} = \dim_{\mathbb{F}_2}H^1(G_S,\mathbb{F}_2[\mathcal{R}_f] \otimes {\wedge^2 J[2]})-\dim_{\mathbb{F}_2}H^1(G_S,\wedge^2 \mathbb{F}_2[\mathcal{R}_f])\\-\dim_{\mathbb{F}_2}H^1(G_S,\mathbb{F}_2[\mathcal{R}_f])+\dim_{\mathbb{F}_2}H^1(G_S,\mathbb{F}_2). 
\end{split}
\end{align}
To compute each of these we use Lemma \ref{Scohomolem}, while we additionally use Lemma \ref{d1lemma} to compute the first term. These computations are simplified by the fact that, through computing dimensions, we obtain the isomorphisms $\mathbb{Q}_f^{(1,2)}\simeq \mathbb{Q}(\alpha,\beta+\gamma)$, $\mathbb{Q}_f^{(2)}\simeq \mathbb{Q}(\alpha+\beta)$ and $\mathbb{Q}_f\simeq \mathbb{Q}(\alpha)$ for distinct roots $\alpha,\beta,\gamma \in \overline{\mathbb{Q}}$ of $f$. In particular, we need the following basic computations:
\begin{itemize}
    \item The number of finite primes above $S$ in $\mathbb{Q}_{f,h}$, $\mathbb{Q}_f^{(2)}$ and $\mathbb{Q}_f$ is $5$, $4$ and $3$, respectively. 
    \item The unit rank of $\mathbb{Q}_{f,h}$, $\mathbb{Q}_f^{(2)}$ and $\mathbb{Q}_f$ is $8$, $6$ and $3$, respectively. 
    \item The relevant $S$-class groups coincide:
    \begin{align*}
    \mathrm{Cl}_S(\mathbb{Q}_{f,h})=\mathrm{Cl}_S(\mathbb{Q}_f^{(2)})=\mathrm{Cl}_S(\mathbb{Q}_f)\simeq \mathbb{Z}/2\mathbb{Z}.
    \end{align*}
\end{itemize}
Accordingly, we obtain from Lemma \ref{d1lemma} that the $\mathbb{F}_2$-dimension of $H^1(G_S,\mathbb{F}_2 [\mathcal{R}_f]\otimes \wedge^2 J[2])$ is at most $29$. Substituting the remaining quantities into \eqref{dxsimp}, we obtain that $d_X\leq 12$. To bound the dimension of $H^1_f(\mathbb{Q},U[2])$, we use the boundary map techniques of \cite{bk2} to compute that this dimension is at most 3. The result then follows from Theorem \ref{mainthm}. 
\end{proof}

\bibliographystyle{alpha}
\bibliography{references}

\end{document}

%% file: header.tex
\usepackage{amsmath,amsfonts,amssymb,amsthm,enumitem,tikz-cd,graphicx,color,verbatim,xspace,bbm}
\usepackage[utf8]{inputenc}
\usepackage{amsmath,amsfonts,amssymb}
\usepackage{amsthm}
\usepackage{color,multicol}
\usepackage{yfonts}
\usepackage{tikz}
\usepackage[new]{old-arrows}
\usepackage{tikz-cd}
\usetikzlibrary{positioning, shapes.geometric, matrix}
\usetikzlibrary{matrix,decorations.pathreplacing, calc, positioning,fit}
\usepackage{scrextend}
\usepackage{adjustbox}
\usepackage{relsize}
\usepackage{fancyhdr}
\usepackage{graphicx}
\usepackage{babel,blindtext}
\usepackage{caption}
\usepackage{mathrsfs} 
\usepackage{scalerel}
\usepackage{mdframed}
\usepackage{centernot}
\usepackage{mathtools}
\usepackage{ stmaryrd }
\usepackage{enumitem}

\numberwithin{equation}{section}
\theoremstyle{plain}
\newcounter{theoremcount}[section]
\makeatletter
\@addtoreset{theoremcount}{section}
\@addtoreset{equation}{section}
\@addtoreset{equation}{section}
\makeatother

\newtheorem{theorem}[theoremcount]{Theorem}
\newtheorem*{theorem*}{Theorem}
\newtheorem{conjecture}[theoremcount]{Conjecture}
\newtheorem{lemma}[theoremcount]{Lemma}
\newtheorem{proposition}[theoremcount]{Proposition}
\newtheorem*{proposition*}{Proposition}

\theoremstyle{definition}
\newtheorem{definition}[theoremcount]{Definition}
\newtheorem{definition-construction}[theoremcount]{Definition/Construction}
\newtheorem{lemma-construction}[theoremcount]{Lemma/Construction}

\newtheorem{remark}[theoremcount]{Remark}

\usepackage[OT2,T1]{fontenc}
\DeclareSymbolFont{cyrletters}{OT2}{wncyr}{m}{n}
\DeclareMathSymbol{\Sha}{\mathalpha}{cyrletters}{"58}

\newtheoremstyle{introstyle}
  {3pt}
  {3pt}
  {\itshape}
  {}
  {\bfseries}
  {.}
  {.5em}
  {}

\theoremstyle{introstyle}
\newtheorem{introtheorem}{Theorem}

\usepackage[colorlinks=true, 
            citecolor=blue, 
            linkcolor=blue, 
            urlcolor=blue, 
            bookmarksopen=true]{hyperref}